\documentclass[12pt,reqno]{amsart}

\usepackage[arrow,matrix,curve]{xy}

\usepackage[dvips]{graphicx} 

\usepackage{amssymb, latexsym, amsmath, amscd, array, hyperref
}

\usepackage{breakurl}   

\newtheorem{theorem}{Theorem}[section]
\newtheorem{lemma}[theorem]{Lemma}
\newtheorem{proposition}[theorem]{Proposition}
\newtheorem{corollary}[theorem]{Corollary}

\theoremstyle{definition}

\numberwithin{equation}{section}

\DeclareMathOperator{\area}{{\rm area}}

\DeclareMathOperator{\sys}{{\rm sys}}
\DeclareMathOperator{\Var}{{\rm Var}}

\newcommand\R {{\mathbb R}}

\newcommand{\RP}{{\mathbb R\mathbb P}}

\numberwithin{equation}{section}

\author{Mikhail G. Katz}\address{M. Katz, Department of Mathematics,
Bar Ilan University, Ramat Gan 5290002
Israel}\email{katzmik@macs.biu.ac.il}

\author{Tahl Nowik}\address{T. Nowik, Department of Mathematics, Bar
Ilan University, Ramat Gan 5290002 Israel}\email{tahl@math.biu.ac.il}

\title [Systolic inequality with remainder in projective plane] {A
systolic inequality with remainder in the real projective plane}

\begin{document}

\thispagestyle{empty}


\begin{abstract}
The first paper in systolic geometry was published by Loewner's
student P. M. Pu over half a century ago.  Pu proved an inequality
relating the systole and the area of an arbitrary metric on the real
projective plane.  We prove a stronger version of Pu's systolic
inequality with a remainder term.
\end{abstract}



\maketitle


\section{Introduction}

Loewner's systolic inequality for the torus and Pu's inequality
\cite{Pu52} for the real projective plane were historically the first
results in systolic geometry.  Great stimulus was provided in 1983 by
Gromov's paper \cite{Gr83}, and later by his book \cite{Gr99}.

Our goal is to prove a strengthened version with a remainder term of
Pu's systolic inequality~$\sys^2(g)\leq\frac\pi2\area(g)$ (for an
arbitrary metric~$g$ on~$\RP^2$), analogous to Bonnesen's
inequality~$L^2 - 4\pi A \geq \pi^2(R-r)^2$, where~$L$ is the length
of a Jordan curve in the plane,~$A$ is the area of the region bounded
by the curve,~$R$ is the circumradius and~$r$ is the inradius.

Note that both the original proof in Pu (\cite{Pu52}, 1952) and the
one given by Berger (\cite{Be65}, 1965, pp.\;299--305) proceed by
averaging the metric and showing that the averaging process decreases
the area and increases the systole.  Such an approach involves a
5-dimensional integration (instead of a 3-dimensional one given here)
and makes it harder to obtain an explicit expression for a remainder
term.  Analogous results for the torus were obtained in \cite{Ho09}
with generalisations in \cite{Ba19} - \cite{Sa11}.

\section{The results}

We define a closed~$3$-dimensional manifold~$M \subseteq \R^3 \times
\R^3$ by setting
\[
M=\{ (v,w) \in \R^3 \times \R^3 \colon \; v\cdot v =1, \ w\cdot w=1, \
v\cdot w=0 \}
\]
where~$v \, \cdot \, w$ is the scalar product on~$\R^3$. We have a
diffeomorphism~$M\to SO(3,\R)$,~$(v,w) \mapsto (v,w,v\times w)$,
where~$v\times w$ is the vector product on~$\R^3$.  Given a
point~$(v,w) \in M$, the tangent space~$T_{(v,w)}M$ can be identified
by differentiating the three defining equations of~$M$ along a path
through~$(v,w)$.  Thus
\[
T_{(v,w)}M=\{ (X,Y) \in \R^3 \times \R^3\colon X\cdot v =0, \ Y\cdot
w=0, \ X\cdot w+Y\cdot v=0 \}.
\]

We define a Riemannian metric~$g_M$ on~$M$ as follows.  Given a point
$(v,w)\in M$, let~$n=v\times w$ and declare the basis~$(0,n) , (n,0) ,
(w,-v)$ of~$T_{(v,w)}M$ to be orthonormal.  This metric is a
modification of the metric restricted to~$M$
from~$\R^3\times\R^3=\R^6$.  Namely, with respect to the Euclidean
metric on~${\R}^6$ the above three vectors are orthogonal and the
first two have length 1.  However, the third vector has Euclidean
length~$\sqrt{2}$, whereas we have defined its length to be 1.  Thus
if~$A\subseteq T_{(v,w)}M$ denotes the span of~$(0,n)$ and~$(n,0)$,
and~$B\subseteq T_{(v,w)}M$ is spanned by~$(w,-v)$, then the metric
$g_M$ on~$M$ is obtained from the Euclidean metric~$g$ on~$\R^6$
(viewed as a quadratic form) as follows:
\begin{equation}
g_M = g\!\downharpoonright_A^{\phantom{I}}+\,\frac12\,
g\!\downharpoonright_B^{\phantom{II}}.
\end{equation}

Each of the natural projections~$p,q \colon M \to S^2$ given
by~$p(v,w)=v$ and~$q(v,w)=w$, exhibits~$M$ as a circle bundle
over~$S^2$.

\begin{lemma}
\label{l101}
The maps~$p$ and~$q$ on~$(M,g_M)$ are Riemannian submersions, over the
unit sphere~$S^2\subseteq\R^3$.
\end{lemma}

\begin{proof}
For the projection~$p$, given~$(v,w) \in M$, the vector~$(0,n)$ as
defined above is tangent to the fiber~$p^{-1}(v)$. The other two
vectors,~$(n,0)$ and~$(w,-v)$, are thus an orthonormal basis for the
subspace of~$T_{(v,w)}M$ normal to the fiber, and are mapped by~$dp$
to the orthonormal basis~$n,w$ of~$T_v S^2$.
\end{proof}

The projection~$p$ maps the fiber~$q^{-1}(w)$ onto a great circle
of~$S^2$.  This map preserves length since the unit vector~$(n,0)$,
tangent to the fiber~$q^{-1}(w)$ at~$(v,w)$, is mapped by~$dp$ to the
unit vector~$n \in T_v S^2$.  The same comments apply when the roles
of~$p$ and~$q$ are reversed.

In the following proposition, integration takes place respectively
over great circles~$C\subseteq S^2$, over the fibers in~$M$,
over~$S^2$, and over~$M$.  The integration is always with respect to
the volume element of the given Riemannian metric.  Since~$p$ and~$q$
are Riemannian submersions by Lemma~\ref{l101}, we can use Fubini's
Theorem to integrate over~$M$ by integrating first over the fibers of
either~$p$ or~$q$, and then over~$S^2$; cf.\;\cite[Lemma\;4]{Cs18}.
By the remarks above, if~$C=p(q^{-1}(w))$ and~$f\colon S^2 \to \R$
then~$\int_{q^{-1}(w)} f\circ p = \int_C f$.

\begin{proposition}
\label{inq}
Given a continuous function~$f\colon S^2 \to \R^+$, we define~$m\in\R$
by setting~$m=\min \{\int_C f\colon C \subseteq S^2 \text{ a great
circle} \}$.  Then
\[
\frac{m^2}{\pi}\leq\frac{1}{4\pi}\bigg(\int_{S^2}f\bigg)^2\leq\int_{S^2}f^2,
\]
where equality in the second inequality occurs if and only if~$f$ is
constant.
\end{proposition}

\begin{proof}
Using the fact that~$M$ is the total space of a pair of Riemannian
submersions, we obtain
\[
\begin{aligned}
\int_{S^2} f & = \int_{S^2}\bigg(\frac{1}{2\pi}\int_{p^{-1}(v)} f\circ
p\bigg)
\\&= \frac{1}{2\pi}\int_M f \circ p 
\\&=
\frac{1}{2\pi}\int_{S^2}\bigg(\int_{q^{-1}(w)} f\circ p\bigg) \\&\geq
\frac{1}{2\pi}\int_{S^2}m =2m,
\end{aligned}
\]
proving the first inequality.  By the Cauchy--Schwarz inequality, we
have
\[
\Big(\int_{S^2} 1 \cdot f \Big)^2 \leq 4 \pi \int_{S^2} f^2,
\]
proving the second inequality.  Here equality occurs if and only if
$f$ and~$1$ are linearly dependent, i.e., if and only if~$f$ is
constant.
\end{proof}

We define the quantity~$V_f$ by setting~$V_f =\int_{S^2} f^2 -
\frac{1}{4\pi}\big(\int_{S^2} f \big)^2$.  Then Proposition \ref{inq}
can be restated as follows.

\begin{corollary}
\label{c1}
Let~$f\colon S^2 \to \R^+$ be continuous.  Then
\[
\int_{S^2} f^2 - \frac{m^2}{\pi} \geq V_f \geq 0,
\]
and~$V_f=0$ if and only if~$f$ is constant.
\end{corollary}

\begin{proof}
The proof is obtained from Proposition~\ref{inq} by noting that~$a\leq
b \leq c$ if and only if~$c-a \geq c-b \geq 0$.
\end{proof}

We can assign a probabilistic meaning to~$V_f$ as follows. Divide the
area measure on~$S^2$ by~$4\pi$, thus turning it into a probability
measure~$\mu$.  A function~$f\colon S^2 \to \R^+$ is then thought of
as a random variable with
expectation~$E_\mu(f)=\frac{1}{4\pi}\int_{S^2}f$.  Its variance is
thus given by
\[
\Var_\mu(f)=E_\mu(f^2)-\big(E_\mu(f)\big)^2 = \frac{1}{4\pi}\int_{S^2}
f^2 - \bigg(\frac{1}{4\pi}\int_{S^2} f\bigg)^2 = \frac{1}{4\pi} V_f.
\]
The variance of a random variable~$f$ is non-negative, and it vanishes
if and only if~$f$ is constant. This reproves the corresponding
properties of~$V_f$ established above via the Cauchy--Schwarz
inequality.

Now let~$g_0$ be the metric of constant Gaussian curvature~$K=1$
on~$\RP^2$.  The double covering~$\rho \colon S^2 \to (\RP^2, g_0)$ is
a local isometry.  Each projective line~$C\subseteq\RP^2$ is the image
under~$\rho$ of a great circle of~$S^2$.

\begin{proposition}
\label{inqP}
Given a function~$f\colon \RP^2 \to \R^+$, we define~$\bar m\in\R$ by
setting~$\bar{m}=\min \{ \int_C f \colon C \subseteq \RP^2 \ \text{a
projective line} \}$.  Then
\[
\frac{2\bar{m}^2}{\pi} \ \leq \ \frac{1}{2\pi}\bigg(\int_{\RP^2} f
\bigg)^2 \ \leq \ \int_{\RP^2} f^2,
\]
where equality in the second inequality occurs if and only if~$f$ is
constant.
\end{proposition}

\begin{proof}
We apply Proposition \ref{inq} to the composition~$f\circ\rho$.  Note
that we have $\int_{\rho^{-1}(C)} f\circ \rho = 2 \int_C f$
and~$\int_{S^2} f \circ \rho = 2 \int_{\RP^2} f$. The condition
for~$f$ to be constant holds since~$f$ is constant if and only
if~$f\circ \rho$ is constant.
\end{proof}

For~$\RP^2$ we define~$\bar{V}_f = \int_{\RP^2} f^2 -
\frac{1}{2\pi}\big(\int_{\RP^2} f \big)^2 =\frac{1}{2}V_{f\circ\rho}$.
We obtain the following restatement of Proposition \ref{inqP}.

\begin{corollary}
\label{c1P} 
Let~$f\colon\RP^2 \to \R^+$ be a continuous function.  Then
\[
\int_{\RP^2} f^2 - \frac{2\bar{m}^2}{\pi} \geq \bar{V}_f \geq 0,
\]
where~$\bar{V}_f=0$ if and only if~$f$ is constant.
\end{corollary}

Relative to the probability measure induced by~$\frac{1}{2\pi}g_0$
on~$\RP^2$, we have~$E(f) = \frac{1}{2\pi}\int_{\RP^2}f$, and
therefore~$\Var(f)=\frac{1}{2\pi}\bar{V}_f$, providing a probabilistic
meaning for the quantity~$\bar{V}_f$, as before.

By the uniformization theorem, every metric~$g$ on~$\RP^2$ is of the
form~$g=f^2 g_0$ where~$g_0$ is of constant Gaussian curvature~$+1$,
and the function~$f\colon\RP^2\to\R^+$ is continuous.  The area of~$g$
is~$\int_{\RP^2} f^2$, and the~$g$-length of a projective line~$C$ is
$\int_C f$.  Let~$L$ be the shortest length of a noncontractible
loop. Then~$L\leq \bar{m}$ where~$\bar{m}$ is defined in
Proposition~\ref{inqP}, since a projective line in~$\RP^2$ is a
noncontractible loop.  Then Corollary~\ref{c1P} implies~$\area(\RP^2,
g) - \frac{2L^2}{\pi} \ \geq \ \bar{V}_f \ \geq \ 0$.
If~$\area(\RP^2, g) = \frac{2L^2}{\pi}$ then~$ \bar{V}_f = 0$, which
implies that~$f$ is constant, by Corollary \ref{c1P}. Conversely, if
$f$ is a constant~$c$, then the only geodesics are the projective
lines, and therefore~$L=c\pi$.  Hence~$\frac{2L^2}{\pi} = 2\pi
c^2=\area(\RP^2)$.  We have thus completed the proof of the following
result strengthening Pu's inequality.

\begin{theorem}
\label{pu}
Let~$g$ be a Riemannian metric on~$\RP^2$.  Let~$L$ be the shortest
length of a noncontractible loop in~$(\RP^2,g)$.
Let~$f\colon\RP^2\to\R^+$ be such that~$g=f^2g_0$ where~$g_0$ is of
constant Gaussian curvature~$+1$.  Then 
\[
\area(g) - \frac{2L^2}{\pi}\geq 2\pi \Var(f),
\]
where the variance is with respect to the probability measure induced
by~$\frac{1}{2\pi} g_0$.  Furthermore,
equality\,~$\area(g)=\frac{2L^2}{\pi}$ holds if and only if~$f$ is
constant.
\end{theorem}

\bibliographystyle{amsalpha}
 
\end{document}